\documentclass[10pt]{amsart}
\usepackage{amssymb}
\usepackage{amsmath}
\usepackage{amsthm}
\usepackage[all]{xy}
\usepackage{mathrsfs}
\usepackage{array}
\usepackage{caption}
\usepackage{lscape}
\usepackage{color}
\usepackage[pdf]{pstricks}

\include{scrload}

\newtheorem{thm}{Theorem}[section]

\newtheorem{prop}[thm]{Proposition}
\newtheorem{lemma}[thm]{Lemma}
\newtheorem{cor}[thm]{Corollary}

\theoremstyle{definition}
\newtheorem{rmk}[thm]{Remark}
\newtheorem{eg}[thm]{Example}
\newtheorem{defn}[thm]{Definition}

\newenvironment{pf}{\begin{proof}}{\end{proof}}

\newcommand{\cA}{\ensuremath{\mathcal{A}}}
\newcommand{\C}{\ensuremath{\mathbb{C}}}

\newcommand{\F}{\ensuremath{\mathbb{F}}}

\newcommand{\bM}{\ensuremath{\mathbb{M}}}

\newcommand{\R}{\ensuremath{\mathbb{R}}}
\newcommand{\cR}{\ensuremath{\mathcal{R}}}

\newcommand{\Z}{\ensuremath{\mathbb{Z}}}

\newcommand{\iso}{\cong}

\newcommand{\la}{\lambda}
\newcommand{\bracket}[1]{ \langle #1 \rangle}

\newcommand{\cl}{\mathrm{cl}}

\newcommand{\map}{\rightarrow}

\DeclareMathOperator{\Ext}{Ext}

\newcommand{\cirrad}{0.15}
\newgray{gridline}{0.75}
\newrgbcolor{grey}{0.5 0.5 0.5}

\newcolumntype{C}{>{$}c<{$}}
\newcolumntype{L}{>{$}l<{$}}
\newcolumntype{R}{>{$}r<{$}}
\newcolumntype{H}{>{\setbox0=\hbox\bgroup$}c<{$\egroup}@{}}

\begin{document}
\title{The $\eta$-inverted $\R$-motivic sphere}
\author{Bertrand J. Guillou}
\address{Department of Mathematics\\ University of Kentucky\\
Lexington, KY 40506, USA}
\email{bertguillou@uky.edu}

\author{Daniel C. Isaksen}
\address{Department of Mathematics\\ Wayne State University\\
Detroit, MI 48202, USA}
\email{isaksen@wayne.edu}
\thanks{The first author was supported by Simons Collaboration Grant 282316.
The second author was supported by NSF grant DMS-1202213.}

\subjclass[2000]{14F42, 55T15, 55Q45}

\keywords{motivic homotopy theory,
stable homotopy group,
$\eta$-inverted stable homotopy group,
Adams spectral sequence}

\begin{abstract}
We use an Adams spectral sequence to calculate the
$\R$-motivic stable homotopy groups after inverting $\eta$.
The first step is to apply a Bockstein spectral sequence in order
to obtain $h_1$-inverted $\R$-motivic $\Ext$ groups, which serve as the
input to the $\eta$-inverted $\R$-motivic Adams spectral sequence.
The second step is to analyze Adams differentials.
The final answer is that the Milnor-Witt $(4k-1)$-stem 
has order $2^{u+1}$, where $u$ is the 2-adic valuation of
$4k$.  This answer is reminiscent of the classical image of $J$.
We also explore some of the Toda bracket structure of the
$\eta$-inverted $\R$-motivic stable homotopy groups.
\end{abstract}

\date{\today}

\maketitle

\section{Introduction}

The first exotic property of motivic stable homotopy groups is that
the  Hopf map $\eta$ is not nilpotent.  This means that 
inverting $\eta$ can be useful for understanding the global structure
of motivic stable homotopy groups.

In \cite{AM} and \cite{GI}, 
the $\eta$-inverted $\C$-motivic 2-completed stable homotopy groups
$\hat{\pi}_{*,*}^\C [\eta^{-1}]$ were explicitly computed to be
\[ 
\F_2[\eta^{\pm 1}][\mu,\varepsilon]/\varepsilon^2.
\]
This result naturally suggests that one should study the structure
of $\eta$-inverted motivic stable homotopy groups over other fields.

In the present article, we consider the $\eta$-inverted $\R$-motivic
2-completed stable homotopy groups $\hat{\pi}_{*,*}^\R [\eta^{-1}]$.
Our main tool is the motivic Adams spectral sequence,
which  takes the form
\[ 
\Ext_{\cA^\R}(\bM_2^\R,\bM_2^\R)[h_1^{-1}] \Rightarrow \hat\pi^\R_{*,*} [\eta^{-1}].
\]
Here $\cA^\R$ is the $\R$-motivic Steenrod algebra, and
$\bM_2^\R$ is the motivic $\F_2$-cohomology of $\R$.
We will exhaustively compute this spectral sequence.

We start by computing the Adams $E_2$-page
$\Ext_{\cA^\R}(\bM^\R_2,\bM^\R_2)[h_1^{-1}]$
using the $\rho$-Bockstein spectral sequence \cite{H} \cite{DI}.
This spectral sequence takes the form
\[
\Ext_{\cA^\C} (\bM_2^\C, \bM_2^\C)[\rho] [h_1^{-1}] \Rightarrow 
\Ext_{\cA^\R} (\bM_2^\R, \bM_2^\R) [h_1^{-1}],
\]
where $\cA^\C$ is the $\C$-motivic Steenrod algebra and
$\bM_2^\C$ is the motivic $\F_2$-cohomology of $\C$.

The input to the $\rho$-Bockstein spectral sequence is completely
known from \cite{GI}.  
In order to deduce differentials, one first observes, as in \cite{DI},
that the groups
$\Ext_{\cA^\R}(\bM^\R_2,\bM^\R_2)[\rho^{-1}, h_1^{-1}]$ 
with $\rho$ and $h_1$ both inverted
are easy
to describe.  Then there is only one pattern of
$\rho$-Bockstein differentials that is consistent with this
$\rho$-inverted calculation.

Having obtained the Adams $E_2$-page
$\Ext_{\cA^\R}(\bM^\R_2,\bM^\R_2)[h_1^{-1}]$,
the next step is to compute Adams differentials.
The extension of scalars functor from $\R$-motivic homotopy
theory to $\C$-motivic homotopy theory induces a map
\[
\xymatrix{
\Ext_{\cA^\R}(\bM_2^\R,\bM_2^\R)[h_1^{-1}] \ar@=[r] \ar[d] & \hat{\pi}^\R_{*,*} [\eta^{-1}] \ar[d] \\
\Ext_{\cA^\C}(\bM_2^\C,\bM_2^\C)[h_1^{-1}] \ar@=[r] & \hat{\pi}^\C_{*,*} [\eta^{-1}]
} 
\]
of Adams spectral sequences.
The bottom Adams spectral sequence is completely understood
\cite{AM} \cite{GI}.
The Adams $d_2$ differentials in the top spectral sequence
can then be deduced by the comparison map.

This leads to a complete description of the 
$h_1$-inverted $\R$-motivic Adams $E_3$-page.  Over $\C$, it turns
out that the $h_1$-inverted Adams spectral sequence collapses at this point.
However, over $\R$, there are higher differentials that we deduce from
manipulations with Massey products and Toda brackets.

In the end, we obtain an explicit description of the
$h_1$-inverted $\R$-motivic Adams $E_\infty$-page, from which we 
can read off the $\eta$-inverted stable motivic homotopy groups over $\R$.

In order to state the result, we need a bit of terminology.
Because $\eta$ belongs to $\hat\pi^\R_{1,1}$, it makes sense to use a grading
that is invariant under multiplication by $\eta$.  The Milnor-Witt
$n$-stem is the direct sum $\bigoplus\limits_p \pi^\R_{p+n,p}$.  
Then multiplication by $\eta$ is an endomorphism 
of the Milnor-Witt $n$-stem.

\begin{thm}
\label{thm:intro}
\mbox{}
\begin{enumerate}
\item
The $\eta$-inverted Milnor-Witt $0$-stem is
$\Z_2 [\eta^{\pm 1}]$, where $\Z_2$ is the ring of 2-adic integers.
\item
If $k > 1$, then the $\eta$-inverted Milnor-Witt $(4k-1)$-stem
is isomorphic to $\Z / 2^{u+1} [ \eta^{\pm 1}]$ as a module over
$\Z_2 [\eta^{\pm 1}]$, where $u$ is the $2$-adic valuation of $4k$.
\item
The $\eta$-inverted Milnor-Witt $n$-stem is zero otherwise.
\end{enumerate}
\end{thm}

For degree reasons, the product structure on
$\hat{\pi}_{*,*}^\R [\eta^{-1}]$ is very simple.
However, there are many interesting Toda brackets.  
We explore much
of the 3-fold Toda bracket structure in this article.
In particular, we will show that all of $\hat{\pi}_{*,*}^\R [\eta^{-1}]$
can be constructed inductively via Toda brackets,
starting from just $2$ and the generator of the
Milnor-Witt 3-stem.

Theorem \ref{thm:intro}
gives a familiar answer.  These groups have the same order as the classical image of $J$.  One might expect a geometric proof that directly compares
the image of $J$ spectrum with the $\eta$-inverted $\R$-motivic sphere.
However, the higher structure of the image of $J$ and
of the $\eta$-inverted $\R$-motivic sphere do not coincide.  This
suggests that a direct geometric proof is not possible.
More specifically,
let $\rho_{15}$ be a generator of the classical
image of $J$ in the 15-stem, i.e., a class detected by $h_0^3 h_4$
in the Adams spectral sequence.
Then the Toda bracket $\langle \rho_{15}, \rho_{15}, 32 \rangle$
contains a generator of the image of $J$ in the $31$-stem, i.e.,
is detected by $h_0^{10} h_5$ in the Adams spectral sequence.
On the other hand, the analogous $\eta$-inverted
$\R$-motivic Toda bracket does not contain a generator
of the Milnor-Witt 31-stem, as shown in Remark \ref{rmk:bracketone}.

The calculation of the $\eta$-inverted $\R$-motivic homotopy groups
leads to questions about $\eta$-inverted motivic homotopy groups
over other fields.  We leave it to the reader to speculate on
the behavior of these $\eta$-inverted groups over other fields.


\section{Preliminaries}\label{sec:background}

\subsection{Notation}

We continue with notation from \cite{DI} as follows:
\begin{enumerate}
\item
$\bM^\R_2=\F_2[\tau,\rho]$ 
is the motivic cohomology of $\R$ with $\F_2$ coefficients, where $\tau$ and $\rho$ have bidegrees $(0,1)$ and $(1,1)$, respectively.
\item
$\bM_2^\C=\F_2[\tau]$ 
is the motivic cohomology of $\C$ with $\F_2$ coefficients, where $\tau$ has bidegree $(0,1)$.
\item
$\cA^\R$ is the mod 2 motivic Steenrod algebra over $\R$.
\item
$\cA^\C$ is the mod 2 motivic Steenrod algebra over $\C$.
\item
$\cA^\cl$ is the classical mod 2 Steenrod algebra.
\item
$\Ext_\R$ is the trigraded ring $\Ext_{\cA^\R}(\bM^\R_2,\bM^\R_2)$.
\item
$\Ext_\C$ is the trigraded ring $\Ext_{\cA^\C}(\bM_2^\C,\bM_2^\C)$.
\item
$\Ext_\cl$ is the trigraded ring $\Ext_{\cA^{\cl}} (\F_2, \F_2)$.
\item
$\cR=\F_2[\rho,h_1^{\pm 1}]$.
\item
$\hat{\pi}_{*,*}^\C$ is the motivic stable homotopy ring of the
2-completed motivic sphere spectrum over $\C$.
\item
$\hat{\pi}_{*,*}^\R$ is the motivic stable homotopy ring of the
2-completed motivic sphere spectrum over $\R$.
\item
The symbols $v_1^4$ and $P$ are used interchangeably for the
Adams periodicity operator.
\end{enumerate}

%
%

\subsection{Grading conventions}

We follow \cite{Istems} in grading $\Ext$
according to  $(s,f,w)$, where:
\begin{enumerate}
\item
$f$ is the Adams filtration, i.e., the homological degree.
\item
$s+f$ is the internal degree, i.e., corresponds to the
first coordinate in the bidegree of the Steenrod algebra.
\item
$s$ is the stem, i.e., the internal degree minus
the Adams filtration.
\item
$w$ is the weight.
\end{enumerate}

We will consider the groups $\Ext_\R[h_1^{-1}]$ in which
$h_1$ has been inverted. 
The degree of $h_1$ is $(1,1,1)$. 
As in \cite{GI}, for this purpose it is convenient to introduce the following gradings whose values are zero for $h_1$.
\begin{enumerate}
\item[(5)]
$mw=s-w$ is the Milnor-Witt degree.
\item[(6)]
$c=s+f-2w$ is the Chow degree.
\end{enumerate}

In order to avoid notational clutter, we will drop $h_1$ from the notation,
except in the formal statements of the main results.  Since $h_1$ is a unit,
no information is lost by doing this.  The correct powers of $h_1$
can always be recovered by checking degrees.

For example, in Lemma \ref{lem:Bock-diff} below, we claim that there
is a differential $d^\rho_3 (v_1^4) = \rho^3 v_2$ in the
$\rho$-Bockstein spectral sequence.
Strictly speaking, this formula is non-sensical because
$d^\rho_3 (v_1^4)$ has Adams filtration $5$ while
$v_2$ has Adams filtration $1$.  The correct full formula is 
$d^\rho_3(v_1^4) = \rho^3 h_1^4 v_2$.

If we are  to ignore multiples of $h_1$,  we must rely
on gradings that take value $0$ on $h_1$.
This explains our preference for Milnor-Witt degree $mw$ and Chow
degree $c$.

\section{The $\rho$-Bockstein spectral sequence}
\label{sctn:rho-Bock}

Recall \cite{H,DI} that the $\rho$-Bockstein spectral sequence takes the form
\[ \Ext_\C[\rho] \Rightarrow \Ext_\R.\]
After inverting $h_1$, by \cite[Theorem 1.1]{GI} this takes the form
\[ \cR[v_1^4,v_2,v_3,\dots] \Rightarrow \Ext_\R[h_1^{-1}],\]
where $\cR=\F_2[\rho,h_1^{\pm 1}]$.
Table \ref{tbl:BockE1Gens} lists the generators of the
Bockstein $E_1$-page.

\begin{table}[ht]
\captionof{table}{Bockstein $E_1$-page generators
\label{tbl:BockE1Gens}}
\begin{center}
\begin{tabular}{LL} 
\hline
(mw,c) & \text{generator} \\ \hline
(0,1) & \rho \\
(4,4) & v_1^4 \\
(3,1) & v_2 \\
(7,1) & v_3 \\
(15,1) & v_4 \\
(2^n-1,1) & v_n \\
 \hline
\end{tabular}
\end{center}
\end{table}

\begin{lemma}
\label{lem:Bock-diff}
In the $\rho$-Bockstein spectral sequence, there are differentials
$d^\rho_{2^n-1}(v_1^{2^n}) = \rho^{2^n-1} v_n$ for $n \geq 2$.
All other non-zero differentials follow from the Leibniz rule.
\end{lemma}

The first few examples of these differentials are
$d_3 ( v_1^4) = \rho^3 v_2$,
$d_7 ( v_1^8) = \rho^7 v_3$, and
$d_{15} ( v_1^{16}) = \rho^{15} v_4$.

\begin{proof}
Inverting $\rho$ induces a map
\[ \xymatrix{
\Ext_\C[h_1^{-1}][\rho] \ar@{=>}[r]^{\rho\text{-Bss}} \ar[d]_{\rho\text{-inv}} & \Ext_\R[h_1^{-1}] \ar[d]^{\rho\text{-inv}} \\
\Ext_\C[h_1^{-1}][\rho^{\pm 1}] \ar@{=>}[r]^{\rho\text{-Bss}} & \Ext_\R[h_1^{-1},\rho^{-1}]
}\]
of $\rho$-Bockstein spectral sequences.
We will establish differentials in the $\rho$-inverted
spectral sequence.  The map of spectral sequences then implies that
the same differentials occur when $\rho$ is not inverted.

Recall \cite[Theorem 4.1]{DI} that there is an 
isomorphism $\Ext_{\mathrm{cl}}[\rho^{\pm 1}] \iso \Ext_\R[\rho^{-1}]$ 
sending the classical element $h_0$ to the motivic element $h_1$. 
Using also that $\Ext_{\cl} [h_0^{-1}] = \F_2[h_0^{\pm 1}]$,
it follows that
$\Ext_\R[h_1^{-1},\rho^{-1}]$ is isomorphic to $\cR[\rho^{-1}]$.
Then the $\rho$-inverted $\rho$-Bockstein spectral sequence takes the form
\[
\xymatrix{
\cR[\rho^{-1}][v_1^4,v_2,v_3,\dots] \ar@{=>}[r]^-{\rho\text{-Bss}} & \cR[\rho^{-1}].
}
\]
Because the target of the 
$\rho$-inverted spectral sequence is very small,
essentially everything must either support a differential
or be hit by a differential.

The $\rho$-Bockstein differentials have degree $(-1,0)$ with respect
to the grading $(mw,c)$ used in Table \ref{tbl:BockE1Gens}.
The elements $\rho^k v_2$ cannot support  differentials because there are
no elements in the Milnor-Witt 2-stem.
The only possibility is that
after inverting $\rho$, there is a $\rho$-Bockstein differential
$d_3( v_1^4) = \rho^3 v_2$.

Then the $\rho$-inverted $E_4$-page is
$\cR[v_1^8, v_3, v_4, \ldots]$.
The elements $\rho^k v_3$ cannot support differentials because
the $\rho$-inverted $E_4$-page has no elements in Milnor-Witt 6-stem.
The only possibility is that 
after inverting $\rho$, there is a $\rho$-Bockstein differential
$d_7 ( v_1^8) = \rho^7 v_3$.

In general, the $\rho$-inverted $E_{2^{n-1}}$-page is
$\cR[v_1^{2^n}, v_n, v_{n+1}, \ldots]$.
The elements $\rho^{k} v_n$ cannot support  differentials
because the $\rho$-inverted $E_{2^{n-1}}$-page has no 
elements in the Milnor-Witt $(2^n-2)$-stem.
The only possibility is that 
after inverting $\rho$, there is a $\rho$-Bockstein differential
$d_{2^n -1 } ( v_1^{2^n}) = \rho^{2^n-1} v_n$.
\end{proof}

The $\rho$-Bockstein $E_\infty$-page can be directly computed
from the differentials
in Lemma \ref{lem:Bock-diff}.
We write $P$ rather than $v_1^4$. 

\begin{prop}
\label{prop:Bock-Einfty}
The 
$\rho$-Bockstein $E_\infty$-page is the
$\cR$-algebra   on the generators 
$P^{2^{n-1} k} v_n$ for $n \geq 2$ and $k \geq 0$ (see
Table~\ref{tbl:BockGens}),
subject to the relations 
\[
\rho^{2^n-1} P^{2^{n-1} k} v_n = 0
\]
for $n \geq 2$ and $k \geq 0$, and
\[
P^{2^{n-1}k}v_n\cdot P^{2^{m-1}j}v_m + P^{2^{n-1}(k+2^{m-n}j)}v_n\cdot v_m =0
\]
for $m \geq n \geq 2$, $k \geq 0$, and $j \geq 0$.
\end{prop}

\begin{table}[ht]
\captionof{table}{Bockstein $E_\infty$-page generators}
\label{tbl:BockGens}
\begin{center}
\begin{tabular}{LLL} 
\hline
(mw,c) & \text{generator} &  \text{$\rho$-torsion} \\ \hline
(0,1) & \rho & \infty \\
(0,0) & h_1 & \infty \\
 (3,1)+k(8,8) & P^{2k}v_2 &3  \\
 (7,1)+k(16,16) & P^{4k}v_3 & 7  \\
(15,1)+k(32,32) & P^{8k}v_4 &  {15}  \\
(2^n-1,1)+k(2^{n+1},2^{n+1}) & P^{2^{n-1}k}v_n &   {2^n-1}  \\
 \hline
\end{tabular}
\end{center}
\end{table}


\begin{rmk} 
\label{rmk:Einfty-basis}
In practice, the relations mean that all $P$'s can be shifted onto the $v_n$ with minimal $n$ in any monomial. 
Thus an 
$\cR$-module basis is given by monomials of the form $P^{2^{n-1} k} v_n \cdot v_{m_1}\cdots v_{m_a}$, where $n\leq m_1\leq \dots\leq m_a$.
For example,
\begin{align*}
P^2 v_2 \cdot P^4 v_2 & = P^6 v_2 \cdot v_2 \\
P^4 v_2 \cdot P^8 v_3 & = P^{12} v_2 \cdot v_3 \\
P^4 v_3 \cdot P^{48} v_5 & = P^{52} v_3 \cdot v_5. 
\end{align*}
\end{rmk}

\section{The Adams $E_2$-page}
\label{sctn:Adams-E2}

Having obtained the $\rho$-Bockstein $E_\infty$-page
in Section \ref{sctn:rho-Bock}, our next task
is to consider hidden extensions in $\Ext_\R [h_1^{-1}]$.
We will show that what you see (in the Bockstein $E_\infty$ page) is what you get (in $\Ext$), as there are no hidden relations. 
This will require some careful analysis of degrees, as well as some
manipulations with Massey products.

\begin{lemma}\label{lem:noRho}
For each $n \geq 2$ and $k \geq 0$, the
element $P^{2^{n-1} k} v_n$ of the Bockstein $E_\infty$-page 
detects a unique element of $\Ext_\R[h_1^{-1}]$.
\end{lemma}

\begin{pf}
We need to show that $P^{2^{n-1} k} v_n$ does not share bidegree
with an element that is divisible by $\rho$.

First suppose that $P^{2^{n-1} k} v_n$ has the same bidegree
as $\rho^b P^{2^{m-1} j} v_m$.
Then
\[
(2^n-1,1) + k(2^{n+1}, 2^{n+1} ) = (2^m-1,1) + j(2^{m+1}, 2^{m+1}) + b(0,1).
\]
Considering only the Milnor-Witt degree, we have
\[
2^n ( 2k+1) = 2^m (2j+1).
\]
Therefore, $n = m$ and $k = j$, so $b = 0$.

Now suppose that $P^{2^{n-1} k} v_n$ shares bidegree with some
element $x$.  By Remark \ref{rmk:Einfty-basis}, we may assume that
$x$ is of the form 
$\rho^b P^{2^{m_1-1} j}v_{m_1} \cdot v_{m_2} \cdots v_{m_a}$,
where $m_1\leq m_2\leq \dots \leq m_a$.
We may also assume that $b \leq 2^{m_1} - 2$, since
$\rho^{2^{m_1} - 1} P^{2^{m-1} j} v_{m_1} = 0$.
Because of the previous paragraph, we may assume that $a \geq 2$.
We wish to show that $b = 0$.

We first show that $n \geq m_a$. 
Let $u(x)$ be the difference $mw - c$.
We have $u(P^{2^{m_1-1} j} v_{m_1})=2^{m_1}-2$ and $u(\rho)=-1$.
Since $b \leq 2^{m_1}-2$, it follows that 
$u(\rho^b P^{2^{m_1-1} j}v_{m_1})\geq 0$. Thus
\[ 2^n-2 =  u(P^{2^{n-1}k} v_n) = 
u( \rho^b P^{2^{m_1-1}j}v_{m_1}) + u( v_{m_2} \cdots v_{m_a}) 
\geq u( v_{m_a}) = 2^{m_a}-2,\]
so that $n\geq m_a$.

Now consider the Milnor-Witt and Chow degrees modulo $4$. 
We have 
\[ (-1, 1) \equiv (-a, a+b) \pmod4, 
\]
so $a\equiv 1\pmod4$  and $b\equiv 0\pmod{4}$. 
Thus either $b=0$, which was what we wanted to show, or $b\geq 4$.

We may now assume that $b \geq 4$.
Since $\rho^4 P^{2j} v_2 = 0$, we must have 
$m_1\geq 3$, so that all $m_i$, and also $n$, are at least $3$.

Next, consider degrees modulo 8. 
Comparing degrees gives
\[ (-1, 1) \equiv (-a , a+b) \pmod8.\]
Thus $b\equiv 0 \pmod 8$, so that $b\geq 8$. 
Since $\rho^8 P^{4j} v_3 = 0$,
we must have that $j_1$, and therefore also $n$ and all other $j_i$, 
to be at least $4$.
This argument can be continued to establish that $b$ and $n$ must be arbitrarily large under the assumption that $b>0$.
\end{pf}

\begin{rmk}
\label{rmk:Ext-names}
Lemma \ref{lem:noRho} allows us to use unambiguously the same notation
$P^{2^{n-1} k} v_n$ for an element of $\Ext_\R[h_1^{-1}]$.
\end{rmk}

\begin{lemma}\label{lem:noRho2}
For each $n \geq 2$ and $k \geq 0$, 
the element 
$P^{2^{n-1} k} v_n \cdot v_n$
of the Bockstein $E_\infty$-page 
detects a unique element of $\Ext_\R [h_1^{-1}]$.
\end{lemma}

\begin{proof}
The Milnor-Witt degree of
$P^{2^{n-1} k} v_n \cdot v_n$ is even, while the 
Milnor-Witt degree of $\rho^b P^{ 2^{m-1} j} v_m$ is odd.
Therefore, these elements cannot share bidegree.

Now suppose that the element
$P^{2^{n-1} k} v_n \cdot v_n$ has the same bidegree as the element
$\rho^b P^{2^{m_1 -1} j} v_{m_1} \cdot v_{m_2} \cdots v_{m_a}$, with
$m_1 \leq m_2 \leq \cdots \leq m_a$, $b \leq 2^{m_1} - 2$, and $a \geq 2$.
The rest of the proof is essentially the same as the proof
of Lemma \ref{lem:noRho}.
Consider $u = mw - c$ to get that $n \geq m_a$.
Then consider congruences $(-2,2)\equiv (-a,a+b)$ modulo higher and
higher powers of $2$ to obtain that $b = 0$.
\end{proof}

\begin{rmk}
The obvious generalization of Lemma \ref{lem:noRho2} to elements 
of the form $P^{2^{n-1} k } v_n \cdot v_m$ is false.
For example,
$P^2 v_2 \cdot v_5$ has the same degree as
$\rho^4 v_3^6$.
\end{rmk}

\begin{rmk}
\label{rmk:noRho}
Lemmas \ref{lem:noRho} and \ref{lem:noRho2}
are equivalent to the claim that there are no $\rho$ multiples in 
the $\rho$-Bockstein $E_\infty$-page in the same bidegrees as
$P^{2^{n-1} k } v_n$ or $P^{2^{n-1} k} v_n \cdot v_n$.
This implies that there are also no $\rho$ multiples in
$\Ext_\R [h_1^{-1}]$ that share bidegree with these elements;
we will need this fact later.
The point is that the $\rho$-Bockstein spectral sequence does not
allow for hidden extensions by $\rho$.
\end{rmk}

\begin{lemma}
\label{lem:Einfty-vanish}
$\Ext_\R [h_1^{-1}]$ is zero 
when the Milnor-Witt stem $mw$ and the Chow degree $c$ are both
equal to $2i$ with $i \geq 1$.
\end{lemma}

\begin{proof}
Under the condition
$mw = c = 2i$,
inspection of Table \ref{tbl:BockE1Gens} shows that
the $\rho$-Bockstein $E_1$-page consists
of products of elements of the form $v_1^4$ or $\rho^{2^n+2^m-4} v_n v_m$.
In the $E_\infty$-page,
$\rho^{2^n+2^m-4} v_n v_m = 0$ since
$\rho^{2^n-1} v_n =0$.
Also, $v_1^{4k}$ supports a differential for all $k \geq 0$.
\end{proof}

\begin{lemma}\label{lem:ExtMassey} 
For each $n\geq 2$, $k\geq 1$ and $m > n$, we have a Massey product 
\[ 
P^{2^{n-1}k + 2^{m-2} }v_n = 
\langle \rho^{2^m-2^n}v_{m},\rho^{2^n-1}, P^{2^{n-1}k}v_n \rangle
\]
in $\Ext_\R[h_1^{-1}]$ with no indeterminacy.
\end{lemma}

\begin{pf} 
The Massey product formula follows from the Bockstein differential 
$d^\rho_{2^{m}-1}(P^{2^{m-2}}) = \rho^{2^m-1}v_m$ 
and May's convergence theorem \cite[Theorem~4.1]{May}. 
There are no crossing Bockstein differentials as all classes are in nonnegative $\rho$-filtration.

The indeterminacy of this bracket is generated by products of the form
$\rho^{2^m-2^n} v_{m} \cdot x$ and
$y \cdot P^{2^{n-1} k} v_n$,
where $x$ and $y$ have appropriate bidegrees.
We showed in Lemma \ref{lem:Einfty-vanish} that 
$0$ is the only possibility for $x$ or $y$.
\end{pf}

\begin{rmk}
\label{rmk:uniqueness}
Lemma \ref{lem:ExtMassey} gives many different Massey products for the same
element.  For example,
\begin{align*}
P^8 v_2 & = \langle \rho^4 v_3, \rho^3, P^6 v_2 \rangle \\
P^8 v_2 & = \langle \rho^{12} v_4, \rho^3, P^4 v_2 \rangle \\
P^8 v_2 & = \langle \rho^{28} v_5, \rho^3, v_2 \rangle.
\end{align*}
May's convergence theorem only says that these Massey products are detected
by the same element of the $\rho$-Bockstein $E_\infty$-page.
By Lemma \ref{lem:noRho}, this $E_\infty$-page element detects
just one element of $\Ext_\R [h_1^{-1}]$, so all of the Massey
products must in fact be equal.
\end{rmk}

\begin{lemma}
\label{lem:ExtMassey2}
For $m > n\geq 2$, there is a Massey product 
\[ 
P^{2^{n-1}k + 2^{m-2} }v_n = 
\langle P^{2^{n-1} k} v_n, \rho^{2^m-2} v_{m}, \rho \rangle
\]
in $\Ext_\R[h_1^{-1}]$ with no indeterminacy.
\end{lemma}

\begin{proof}
The Massey product formula follows from the Bockstein differential 
$d^\rho_{2^{m}-1}(P^{2^{m-2}}) = \rho^{2^m-1}v_m$ 
and May's convergence theorem \cite[Theorem~4.1]{May}. 
There are no crossing Bockstein differentials as all classes are in nonnegative $\rho$-filtration.

The indeterminacy of this bracket is generated by products of the form
$P^{2^{n-1} k} v_n \cdot x$ and
$y \cdot \rho$.
We showed in Lemma \ref{lem:Einfty-vanish} that 
$0$ is the only possibility for $x$.
We observed in Remark \ref{rmk:noRho}
that $y \cdot \rho$ must be zero because there are no multiples of
$\rho$ in the appropriate bidegree.
\end{proof}

The relations in the Bockstein $E_\infty$-page given in 
Proposition \ref{prop:Bock-Einfty} may lift to
$\Ext_\R [h_1^{-1}]$ with additional terms that are multiples of
$\rho$.
In other words, there may be hidden relations in the
Bockstein spectral sequence.
For example, for degree reasons it is possible that
$P^2 v_2 \cdot P^{16} v_5 + P^{18} v_2 \cdot v_5$
equals $\rho^4 P^{16} v_3 \cdot v_3^5$.
Proposition \ref{prop:no-hidden} shows that there are no such hidden
terms in the relations in $\Ext_\R [h_1^{-1}]$.

\begin{prop}
\label{prop:no-hidden}
There are no hidden relations in the Bockstein spectral sequence.
\end{prop}

\begin{proof}
There can be no hidden $\rho$ extensions in the $\rho$-Bockstein
spectral sequence, so we need only compute the products
$P^{2^{n-1}k}v_n \cdot P^{2^{m-1}j}v_m$ 
in $\Ext_\R[h_1^{-1}]$ for $m \geq n$.

Lemma \ref{lem:ExtMassey} implies that 
$P^{2^{n-1}k}v_n \cdot P^{2^{m-1}j}v_m$ 
equals
\[
P^{2^{n-1} k} v_n \langle \rho^{2^m} v_{m+1}, \rho^{2^m-1}, 
P^{2^{m-1} (j-1) } v_m \rangle.
\]
Shuffle to obtain
\[
\langle P^{2^{n-1} k} v_n, \rho^{2^m} v_{m+1}, \rho^{2^m-1} \rangle
P^{2^{m-1} (j-1) } v_m.
\]
This expression is contained in
\[
\langle P^{2^{n-1} k} v_n, \rho^{2^{m+1} - 2} v_{m+1}, \rho \rangle
P^{2^{m-1} (j-1) } v_m,
\]
which equals
$P^{2^{n-1} k + 2^{m-1}} v_n \cdot P^{2^{m-1} (j-1)} v_m$
by Lemma \ref{lem:ExtMassey2}.

By induction, 
$P^{2^{n-1}k}v_n \cdot P^{2^{m-1}j}v_m$ 
equals 
$P^{2^{n-1}(k+ 2^{m-n} j) }v_n \cdot v_m$.
\end{proof}

\begin{thm}  
$\Ext_\R[h_1^{-1}]$ is the
$\cR$-algebra   on the generators 
$P^{2^{n-1} k} v_n$ for $n \geq 2$ and $k \geq 0$ (see
Table~\ref{tbl:BockGens}),
subject to the relations 
\[
\rho^{2^n-1} P^{2^{n-1} k} v_n = 0
\]
for $n \geq 2$ and $k \geq 0$, and
\[
P^{2^{n-1}k}v_n\cdot P^{2^{m-1}j}v_m + P^{2^{n-1}(k+2^{m-n}j)}v_n\cdot v_m =0
\]
for $m \geq n \geq 2$, $k \geq 0$, and $j \geq 0$.
\end{thm}

\begin{pf}
This follows immediately from Propositions \ref{prop:Bock-Einfty}
and \ref{prop:no-hidden}.
\end{pf}

\begin{rmk} 
\label{rmk:Ext-basis}
Analogously to Remark \ref{rmk:Einfty-basis},
an 
$\cR$-module basis 
for $\Ext_\R[h_1^{-1}]$
is given by monomials of the form $P^{2^{n-1} k} v_n \cdot v_{m_1}\cdots v_{m_a}$, where $n\leq m_1\leq \dots\leq m_a$.
\end{rmk}

\section{Adams differentials}
\label{sctn:AdamsDiffs}

Before computing with the $h_1$-inverted $\R$-motivic Adams spectral
sequence, we will consider convergence.  
A priori, there could be an infinite
family of homotopy classes linked together by infinitely many hidden
$\eta$ multiplications. 
These classes would not be detected in $\Ext_\R [h_1^{-1}]$.
Lemma \ref{lem:Ext-finite} implies that this cannot occur for degree 
reasons.

\begin{lemma}
\label{lem:Ext-finite}
Let $m > 0$ be a fixed Milnor-Witt stem.
There exists a constant $A$ such that
$\Ext_\R^{(s,f,w)}$ vanishes when $s - w = m$, $s \neq 0$,
$f > A$, and $f > s + 1$.
\end{lemma}

Lemma \ref{lem:Ext-finite} can be restated in the following more casual
form.
Within a fixed Milnor-Witt stem, there exists a horizontal line
and a line of slope 1 such that
$\Ext_\R$ vanishes in the region above both lines,
except in the $0$-stem.
Figure \ref{fig:vanishing} depicts the shape of the vanishing region.

\begin{center}
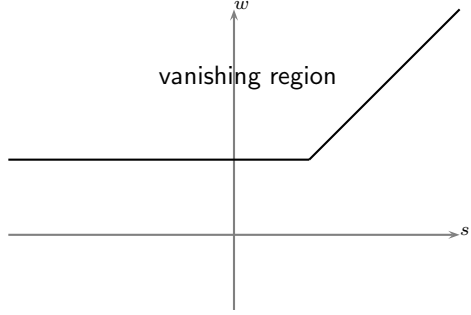
\begin{figure}[ht]
\caption{The vanishing region in a Milnor-Witt stem \label{fig:vanishing}}
\begin{pspicture}(-3,-1)(3,3.5)
\psline[linecolor=grey]{->}(-3,0)(3,0)
\psline[linecolor=grey]{->}(0,-1)(0,3)

\psline(1,1)(3,3)
\psline(-3,1)(1,1)

\put(-1,2){\textsf{vanishing region}}
\scriptsize
\put(3,0){$s$}
\put(0,3){$w$}
\end{pspicture}
\end{figure}
\end{center}

\begin{proof}
This argument occurs in $\Ext_\R$, where
$h_1$ has not been inverted.

As explained in \cite[Theorem 4.1]{DI},
the elements in the $m$-stem of the classical $\Ext$ groups
$\Ext_{\cl}$ correspond to
elements of $\Ext_\R$ in the Milnor-Witt $m$-stem 
that remain non-zero after $\rho$ is inverted, i.e., that support
infinitely many multiplications by $\rho$.
Each stem of $\Ext_{\cl}$ is
finite except for the $0$-stem.  For $m > 0$, 
choose $A$ to be larger than the Adams filtrations
of all of the elements in the $m$-stem of $\Ext_{\cl}$.
Then $A$ is larger than the Adams filtrations of every
element of $\Ext_\R$
in the Milnor-Witt $m$-stem that remain non-zero after
$\rho$ is inverted.

Let $x$ be a non-zero element of $\Ext_\R^{(s,f,w)}$
such that $s - w = m$, $f > A$, and $f > s + 1$. 
We will show that $s$ must equal zero.

The choice of $A$ guarantees that $x$ is annihilated by some positive power 
of $\rho$.
Suppose that $\rho^k x = 0$ but $\rho^{k-1} x$ is non-zero, for some $k>0$.
Then there must be a differential in the
$\rho$-Bockstein spectral sequence 
of the form
$d_k (y ) = \rho^k x$, where $y$ is an element
of $\Ext_\C$ in degree $(s-k+1, f-1, w- k)$.

The argument from \cite{Adams1} establishes a vanishing line of slope
$1$ in $\Ext_\C$.  
The conditions $s > f + 1$ and $k > 0$ imply that
the element $y$ lies strictly above this vanishing line,
so it must be of the form $\tau^a h_0^b$ with $b \geq 1$.
The only $\rho$-Bockstein differentials on such classes are $d_1( \tau^{2c+1} h_0^b ) =
\rho \tau^{2c} h_0^{b+1}$, which
implies that $x$ must be of the form $\tau^{2c} h_0^b$.  
This shows that $s = 0$.
\end{proof}

The $h_1$-inverted motivic Adams spectral sequence over $\C$ was 
studied in \cite{GI,AM}.  It takes the form
\[ 
\F_2[h_1^{\pm 1}, P,v_2,v_3,\dots] \Rightarrow 
\hat{\pi}^\C_{*,*}[\eta^{-1}],
\]
where $\hat{\pi}^\C_{*,*}[\eta^{-1}]$ is the $\eta$-inverted
motivic stable homotopy ring of the 2-completed motivic 
sphere spectrum over $\C$.
This spectral sequence has differentials
\[
\label{d2diffs} 
d_2(P^kv_n) = P^k v_{n-1}^2
\]
for all $k \geq 0$ and all $n \geq 3$.

\begin{lemma}
\label{lem:Adams-d2}
In the $h_1$-inverted $\R$-motivic Adams spectral sequence, there are 
differentials
\[
d_2(P^{2^{n-1} k} v_n ) = P^{2^{n-1} k} v_{n-1}^2
\]
for all $k \geq 0$ and all $n \geq 3$.
\end{lemma}

\begin{proof}
There is an extension of scalars functor from $\R$-motivic homotopy theory
to $\C$-motivic homotopy theory.  This functor induces a map
\[
\xymatrix{
\Ext_{\cA^\R}(\bM_2^\R,\bM_2^\R)[h_1^{-1}] \ar@=[r] \ar[d] & \hat{\pi}^\R_{*,*} [\eta^{-1}] \ar[d] \\
\Ext_{\cA^\C}(\bM_2^\C,\bM_2^\C)[h_1^{-1}] \ar@=[r] & \hat{\pi}^\C_{*,*} [\eta^{-1}]
} 
\]
from the $\R$-motivic Adams spectral sequence to the $\C$-motivic Adams
spectral sequence.  This map takes $\rho$ to zero.

The map of spectral sequences implies that the $\R$-motivic
Adams differential 
$d_2(P^{2^{n-1} k} v_n )$ equals $P^{2^{n-1} k} v_{n-1}^2$ plus 
terms that are divisible by $\rho$.
Lemma \ref{lem:noRho2} implies that there are no possible additional
terms in the relevant bidegree.
\end{proof}

We record the following simple computation, as we will employ it several times.

\begin{lemma}\label{lem:basic} Let $S$ be an $\F_2$-algebra. Let $B = S[w_1,w_2,\dots]$ be a polynomial ring in infinitely many variables, and define a differential on $B$ by $\partial (w_n) = w_{n-1}^2$ for $n\geq 2$. Then $\mathrm{H}^*(B,\partial) \iso S[w_1]/w_1^2$.
\end{lemma}

Lemma~\ref{lem:basic} implies, for example, that the $h_1$-inverted 
$\C$-motivic Adams $E_3$-page is $\F_2[h_1^{\pm 1}, P,v_2]/v_2^2$.

\begin{prop} 
\label{prop:Adams-E3}
The $h_1$-inverted $\R$-motivic 
Adams $E_3$-page 
is free as an $\cR$-module on the generators 
listed in Table~\ref{tbl:E3gen}
for $n\geq 2$, $k \geq 0$, and $j \geq 0$.
Almost all products of these generators are zero, except
that 
\[
P^{4k} v_2 \cdot P^{4j+2} v_2 = P^{4k+4j+2} v_2^2
\]
and
\[
\rho^{2^{n-1} - 1} P^{2^{n-1} \cdot 2k} v_n \cdot 
\rho^{2^{n-1} - 1} P^{2^{n-1} (2j+1)} v_n =
\rho^{2^n - 2} P^{2^{n-1} (2k+2j+1)} v_n^2
\]
for $n \geq 3$.
\end{prop}

\begin{table}[ht]
\captionof{table}{$\cR$-module generators for the Adams $E_3$-page}
\label{tbl:E3gen}
\begin{center}
\begin{tabular}{LLL} 
\hline
(mw,c) & \text{generator} & \text{$\rho$-torsion} \\ \hline
(0,0) & 1 & \infty \\
 (3,1)+k(8,8) & P^{2k}v_2 &3  \\
 (7,4)+k(16,16) & \rho^3P^{4k}v_3 & 4  \\
(15,8)+k(32,32) & \rho^7 P^{8k}v_4 &  8  \\
(2^n-1,2^{n-1})+k(2^{n+1},2^{n+1}) & \rho^{2^{n-1}-1}P^{2^{n-1}k}v_n &   2^{n-1}  \\
\hline
 (6,2)+(2j+1)(8,8) & P^{2(2j+1)}v_2^2 & 3 \\
 (14,2)+(2j+1)(16,16) & P^{4(2j+1)}v_3^2 & 7  \\
 (30,2)+(2j+1)(32,32) & P^{8(2j+1)}v_4^2 & 15  \\
(2^{n+1}-2,2)+(2j+1)(2^{n+1},2^{n+1}) & P^{2^{n-1}(2j+1)}v_n^2 &   2^{n}-1  \\
 \hline
\end{tabular}
\end{center}
\end{table}

\begin{rmk}
\label{rmk:Adams-E3}
The relations in Proposition \ref{prop:Adams-E3} are just the ones
that are obvious from the notation.  For example,
\begin{align*}
v_2 \cdot P^2 v_2 & = P^2 v_2^2 \\
\rho^3 P^4 v_3 \cdot \rho^3 P^8 v_3 & = \rho^6 P^{12} v_3^2.
\end{align*}
\end{rmk}

\begin{pf}
Let $\Ext\langle k, b \rangle$ be the $\F_2[h_1^{\pm 1}]$-submodule of
the $h_1$-inverted $\R$-motivic Adams $E_2$-page on generators 
of the form $\rho^b P^k v_{m_1} v_{m_2} \cdots v_{m_a}$ such that
$m_1 \leq m_2 \leq \cdots \leq m_a$.
Note that $b \leq 2^{m_1} - 2$ in this situation, since
$\rho^{2^{m_1} - 1} P^k v_{m_1} = 0$.
Also, $k$ must be a multiple of $2^{ {m_1} - 1}$.
By Lemma \ref{lem:Adams-d2},
each $\Ext \langle k, b \rangle$
is a differential graded submodule.
Thus it suffices to compute the cohomology of each $\Ext\langle k,b \rangle$. 

We start with $\Ext\langle 0, b\rangle$, which is equal to
to $\rho^b \cdot \F_2 [ h_1^{\pm 1}, v_m, v_{m+1}, \ldots ]$
as a differential graded $\F_2[h_1^{\pm 1}]$-module,
where $m$ is the smallest integer such that $b \leq 2^m - 2$.
Now Lemma \ref{lem:basic} implies that
$H^* ( \Ext \langle 0, b \rangle, d_2 )$ is a free
$\F_2[h_1^{\pm 1}]$-module on two generators $\rho^b$ and $\rho^b v_m$.

So far, we have demonstrated that the powers of $\rho$ and the elements
\[
v_2, \rho v_2, \rho^2 v_2, \rho^3 v_3, \ldots \rho^6 v_3, \rho^7 v_4, \ldots
\]
are present in the $h_1$-inverted $\R$-motivic Adams $E_3$-page.

The module $\Ext \langle k, b \rangle$ is zero when $k$ is odd.

Now assume that $k$ is equal to $2$ modulo $4$.
If $b \leq 2$, then
$\Ext \langle k, b \rangle$ is equal to
$\rho^b P^k v_2 \cdot \F_2 [ h_1^{\pm 1}, v_2, v_3, \ldots ]$
as a differential graded $\F_2[h_1^{\pm 1}]$-module.
Lemma \ref{lem:basic} implies that
$H^* (\Ext \langle k, b \rangle, d_2)$ is a free
$\F_2[h_1^{\pm 1}]$-module on two generators
$\rho^b P^k v_2$ and $\rho^b P^k v_2^2$.
If $b \geq 3$, then
$\Ext \langle k, b \rangle$ is zero because $\rho^3 P^k v_2 = 0$.

We have now shown that the elements
\[
P^k v_2, \rho P^k v_2, \rho^2 P^k v_2,
P^k v_2^2, \rho P^k v_2^2, \rho^2 P^k v_2^2
\]
are present in the $h_1$-inverted $\R$-motivic Adams
$E_3$-page for all $k$ congruent to $2$ modulo $4$.

Next assume that $k$ is equal to $4$ modulo $8$.
If $b \leq 2$, then
$\Ext \langle k, b \rangle$ 
is the free $\F_2 [h_1^{\pm 1}]$-module on generators
$\rho^b P^k v_{m_1} \cdots v_{m_a}$ such that
$m_1$ equals $2$ or $3$, and
$m_1 \leq \cdots \leq m_a$.
There is a short exact sequence
\[
\Ext \langle k, b \rangle \map
\rho^b P^k \cdot \F_2 [h_1^{\pm 1}, v_2, v_3, \ldots ] \map
\rho^b P^k \cdot \F_2 [h_1^{\pm 1}, v_4, v_5, \ldots ].
\]
By Lemma \ref{lem:basic}, 
the homology of the middle term has two generators
$\rho^b P^k$ and $\rho^b P^k v_2$,
while the homology of the right term has two generators
$\rho^b P^k$ and $\rho^b P^k v_4$.
Analysis of the long exact sequence in homology shows that 
$H^*(\Ext \langle k, b \rangle, d_2)$ has two generators
$\rho^b P^k v_2$ and $\rho^b P^k v_3^2$.

Now assume that $3 \leq b \leq 6$.
Since $\rho^b P^k v_2 = 0$, we get that
$\Ext \langle k, b\rangle$ is
equal to $\rho^b P^k v_3 \cdot \F_2[h_1^{\pm 1}, v_3, v_4, \ldots ]$.
Lemma \ref{lem:basic} implies that
$H^* (\Ext \langle k, b \rangle, d_2)$ is a free
$\F_2[h_1^{\pm 1}]$-module on two generators
$\rho^b P^k v_3$ and $\rho^b P^k v_3^2$.

Finally, if $b \geq 7$, then
$\Ext \langle k, b \rangle$ is zero because $\rho^7 P^k v_2 = 0$
and $\rho^7 P^k v_3 = 0$.  This finishes the argument
when $k$ is equal to $4$ modulo $8$, and we have shown that
$\Ext_\R [h_1^{\pm 1}]$ contains the elements
\begin{align*}
& P^k v_2, \rho P^k v_2, \rho^2 P^k v_2 \\
& \rho^3 P^k v_3, \ldots, \rho^6 P^k v_3 \\
& P^k v_3^2, \rho P^k v_3^2, \ldots, \rho^6 P^k v_3^2.
\end{align*}

Analysis of the other cases is the same as the argument for
$k = 4$ modulo $8$.
The details depend on the value of $k$ modulo $2^i$
and inequalities of the form $2^j - 1 \leq b \leq 2^{j+1} - 2$.
In each case there is a short exact sequence of 
differential graded modules whose first term is $\Ext \langle k, b \rangle$
and whose other two terms have homology that is computed
by Lemma \ref{lem:basic}.
\end{pf}

We have now calculated the $h_1$-inverted $\R$-motivic $E_3$-page. 
This $E_3$-page is displayed in Figure \ref{fig:E3}.
Beware that the grading on this chart is not the same as in a standard
Adams chart.  The Milnor-Witt stem $mw = s-w$ is plotted on the horizontal
axis, while the Chow degree $c = s+f-2w$ is plotted on the vertical axis.
As a result, an Adams $d_r$ differential has slope $-r+1$, rather than
slope $-r$.  Vertical lines in Figure \ref{fig:E3} represent
multiplications by $\rho$.

Our next goal is to establish the Adams $d_3$ differentials.
Inspection of Figure \ref{fig:E3} reveals that
the only possible non-zero $d_3$ differentials
might be supported on elements of the form $\rho^b P^{2^{n-1} k} v_n$
for $n \geq 4$.  In fact, these differentials all occur,
as indicated in Figure \ref{fig:E3} by lines that go left one unit
and up two units.
We will establish these $d_3$ differentials by first proving 
a homotopy relation in Lemma \ref{lem:rhoDiv}.

\begin{lemma}\label{lem:rhoDiv} 
For each $n\geq 2$ and $j\geq 0$, the element 
$P^{2^{n-1}(2j+1)} v_n^2$ is a permanent cycle
that detects a $\rho$-divisible element of the
$\eta$-inverted $\R$-motivic homotopy groups.
\end{lemma}

\begin{pf}
Inspection of Figure \ref{fig:E3} shows that 
$P^{2^{n-1}(2j+1)} v_n^2$ cannot support a differential.

Lemma \ref{lem:ExtMassey2} implies that
$P^{2^{n-1}(2j+1)} v_n^2$ belongs to the Massey product
$\langle \rho, \rho^{2^{n+1}-2} v_{n+1}, P^{2^n j} v_n^2 \rangle$
in $\Ext_\R[h_1^{-1}]$.
In fact, the Massey product has no indeterminacy because of 
Remark \ref{rmk:noRho} and Lemma \ref{lem:Einfty-vanish}.

We will now apply Moss's convergence theorem \cite[Theorem~1.2]{Moss}
to this Massey product.
There is an Adams differential $d_2(P^{2^n j} v_{n+1}) = P^{2^n j} v_n^2$,
so $P^{2^n j} v_n^2$ detects the homotopy element $0$.
By inspection of Figure \ref{fig:E3},
$\rho^{2^{n+1} - 2} v_{n+1}$ is a permanent cycle;
let $\alpha$ be a homotopy element detected by it.
Moreover, $\rho \alpha$ is zero in homotopy because there are no
classes in higher filtration that could detect it.

Moss's convergence theorem says that
the Toda bracket $\langle \rho, \alpha, 0 \rangle$
contains an element that is detected by $P^{2^{n-1}(2j+1)} v_n^2$.
This Toda bracket consists entirely of multiples of $\rho$.
\end{pf}

\begin{lemma}\label{lem:d3}
\[ d_3( \rho^{2^{n-1}-1} P^{2^{n-1}k} v_n) = P^{2^{n-3}+2^{n-1}k} v_{n-2}^2\]
for $n\geq 4$. 
\end{lemma}

\begin{pf}
Lemma \ref{lem:rhoDiv} shows that
$P^{2^{n-3} + 2^{n-1} k} v_{n-2}^2$ detects a class that is
divisible by $\rho$.
By inspection of Figure \ref{fig:E3},
there are no classes in lower filtration.
Therefore,
$P^{2^{n-3} + 2^{n-1} k} v_{n-2}^2$ must detect zero, i.e.,
must be hit by a differential.
It is apparent from Figure \ref{fig:E3} 
that there is only one possible differential.
\end{pf}

Lemma \ref{lem:Adams-higher}
describes the higher Adams differentials.

\begin{lemma}
\label{lem:Adams-higher}
For $n \geq r+1$ and $r \geq 3$,
\[
d_r ( \rho^{2^n - 2^{n-r+2} - r + 2} P^{2^{n-1} k} v_n ) =
P^{2^{n-1} k + 2^{n-2} - 2^{n-r}} v_{n-r+1}^2.
\]
\end{lemma} 

\begin{proof}
The proof is essentially the same as the proof of 
Lemma \ref{lem:d3}.
In Milnor-Witt stem congruent to $2$ modulo $4$,
Lemma \ref{lem:rhoDiv} implies that every homotopy element
is divisible by $\rho$.
This implies that they must all be hit by differentials.
Figure \ref{fig:E3} indicates that there is just one possible pattern
of differentials.
\end{proof}

From Lemma \ref{lem:Adams-higher},
it is straightforward to derive the $h_1$-inverted Adams
$E_\infty$-page, as shown in Figure \ref{fig:Einfty}.

\begin{prop}\label{prop:EinfGen} The $h_1$-inverted Adams $E_\infty$-page is the $\cR$-module on generators given in Table~\ref{tbl:Einfgen} 
for $n\geq 2$.
\end{prop}

\begin{table}[ht]
\captionof{table}{$\cR$-module generators for the Adams $E_\infty$-page}
\label{tbl:Einfgen}
\begin{center}
\begin{tabular}{LLL} 
\hline
(mw,c) & \text{generator} & \text{$\rho$-torsion} \\ \hline
(0,0) & 1 &\infty \\
(3,1)+k(8,8) & P^{2k}v_2 & 3  \\
(7,4)+k(16,16) & \rho^3P^{4k}v_3 & 4  \\
(15,11) + k(32,32) & \rho^{10}P^{8k}v_4 & 5  \\
(2^n-1,2^n-n-1)+k(2^{n+1},2^{n+1}) & \rho^{2^n-n-2}P^{2^{n-1}k}v_n & n+1  \\
 \hline
\end{tabular}
\end{center}
\end{table}

\section{$\eta$-inverted homotopy groups}

From the $h_1$-inverted Adams $E_\infty$-page, it is a short
step to the $\eta$-inverted stable homotopy ring.
First we must choose generators.

\begin{defn}
\label{defn:gen}
For $k \geq 0$ and $n \geq 2$,
let $P^{2^{n-1} k} \lambda_n$ be an element of the Milnor-Witt
$(2^{n+1} k + 2^n - 1)$-stem that is detected by
$\rho^{2^n - n - 2} P^{2^{n-1} k} v_n$.
\end{defn}

There are choices in these definitions, which are measured by
Adams $E_\infty$-page elements in higher filtration.
For example, there are four possible choices for $\lambda_2$ because of the 
presence of $\rho v_2$ and $\rho^2 v_2$ in higher filtration.  

\begin{thm}
\label{thm:ring}
As a $\Z_2[\eta^{\pm 1}]$-module,
the $\eta$-inverted $\R$-motivic stable homotopy ring is
generated by $1$ and $P^{2^{n-1} k}\lambda_n$ for $n\geq 2$ and $k\geq 0$. The generator $P^{2^{n-1} k}\la_n$ lies in the Milnor-Witt 
$(2^{n+1} k +2^n-1)$-stem and is annihilated by $2^{n+1}$.
All products are zero, except for those involving $2$ or $\eta$.
\end{thm}

\begin{proof}
In the $\eta$-inverted stable homotopy ring,
$\rho$ and $2$ differ by a unit because $\rho \eta^2 = - 2 \eta$ \cite{Morel}. Therefore, the $\rho$-torsion information given in Proposition~\ref{prop:EinfGen} translates to $2$-torsion information in homotopy.

Except for $1$, all $\Z_2 [\eta^{\pm 1}]$-module generators
lie in Milnor-Witt stems that are congruent to $3$ modulo $4$.
Therefore, such generators must multiply to zero.
\end{proof}

Table~\ref{tbl:Htpygen} lists all generators through the Milnor-Witt $63$-stem.  The table also identifies Toda brackets that contain each generator.
These Toda brackets are computed in Section \ref{sctn:Toda}.

\begin{table}[ht]
\captionof{table}{$\Z_2[\eta^{\pm 1}]$-module generators for $\hat\pi^{\R}_{*,*}[\eta^{-1}]$
\label{tbl:Htpygen}}
\begin{center}
\begin{tabular}{LLLLLL} 
\hline
mw & \text{$E_\infty$} & \text{$\hat\pi^\R_{*,*}[\eta^{-1}]$} & \text{$2^k$-torsion} & \text{bracket} & \text{indeterminacy} \\ \hline
0 & 1 & 1 &\infty & & \\
3 & v_2 & \la_2 & 3 & &  \\
7 & \rho^3v_3 & \la_3 & 4 &\langle 2^3, \lambda_2, \lambda_2 \rangle &
   2^3\la_3 \\
11 & P^2 v_2 & P^2\la_2 & 3 & \bracket{ 2^4, \la_3, \la_2} & \\
15 & \rho^{10}v_4 & \la_4 & 5 & \bracket{ 2^3, \la_2, P^2\la_2} 
   & 2^3 \la_{4}\\
19 & P^4 v_2 & P^4 \la_2 & 3 & \bracket{ 2^5, \la_{4}, \la_2} & \\
23 & \rho^3P^4 v_3 & P^4 \la_3 & 4 & \bracket{ 2^5, \la_{4}, \la_3} & \\
27 & P^6 v_2 & P^6 \la_2 & 3 & \bracket{ 2^5, \la_{4}, P^2\la_{2}} & \\
31 & \rho^{25} v_5 & \la_5 & 6 & \bracket{ 2^3, \la_2, P^6\la_{2}} 
   & 2^3 \la_{5}\\
35 & P^8 v_2 & P^8 \la_2 & 3 & \bracket{2^6 ,\la_5 ,\la_2 } & \\
39 & \rho^3 P^8 v_3 & P^8 \la_3 & 4 & \bracket{2^6, \la_5, \la_3 } & \\
43 & P^{10} v_2 & P^{10} \la_2 & 3 & \bracket{2^6, \la_5, P^2\la_2} & \\
47 & \rho^{10}P^8 v_4 & P^8 \la_4 & 5 & \bracket{2^6, \la_5, \la_4} & \\
51 & P^{12}v_2 & P^{12}\la_2 & 3 & \bracket{2^6, \la_5, P^4\la_2} & \\
55 & \rho^3P^{12}v_3 & P^{12}\la_3 & 4 & \bracket{2^6, \la_5, P^4\la_3} &\\
59 & P^{14}v_2 & P^{14}\la_2 & 3 & \bracket{2^6, \la_5, P^6\la_2} & \\
63 & \rho^{56}v_6 & \la_6 & 7 & \bracket{2^3, \la_2, P^{14}\la_2} 
   & 2^3 \la_6\\
 \hline
\end{tabular}
\end{center}
\end{table}

Table \ref{tbl:Htpygen} also reveals a pattern that matches the classical
image of $J$.

\begin{cor}
\label{cor:ImJ}
If $k > 1$, then the $\eta$-inverted Milnor-Witt $(4k-1)$-stem
is isomorphic to $\Z / 2^{u+1} [ \eta^{\pm 1}]$ as a module over
$\Z_2 [\eta^{\pm 1}]$, where $u$ is the $2$-adic valuation of $4k$.
\end{cor}


\section{Toda brackets}
\label{sctn:Toda}

Even though its primary multiplicative structure 
is uninteresting, 
the $\eta$-inverted $\R$-motivic 
stable homotopy ring has rich higher structure in the
form of Toda brackets.  We will explore some of the 3-fold Toda bracket
structure.
In particular, we will show that all of the generators can be 
inductively
constructed via Toda brackets, starting from just $2$ and $\lambda_2$.
Table~\ref{tbl:Htpygen} lists one possible Toda bracket decomposition
for each generator up to the Milnor-Witt 63-stem.

We observed in the proof of Theorem \ref{thm:ring} that
the element $\rho$ of the Adams $E_\infty$-page detects
the element $2$ of the $\eta$-inverted
stable homotopy ring.  We will use this fact frequently
in the following results.

\begin{lemma}\label{lem:bracketone} 
The Toda bracket $\bracket{2^3,\la_2,\la_2}$ 
contains an element detected by $\rho^3 v_3$ in the Milnor-Witt $7$-stem,
and its indeterminacy is detected by $\rho^6 v_3$.
\end{lemma}

\begin{pf}
Moss's convergence theorem \cite[Theorem~1.2]{Moss}
and the differential $d_2(v_3) = v_2^2$ show that
$\bracket{2^3, \la_2, \la_2}$
is detected by $\rho^3 v_3$.

The indeterminacy follows from the facts that there are no 
multiples of $\lambda_2$ and that there is a unique multiple
of $2^{3}$ in the Milnor-Witt $7$-stem.
\end{pf}


\begin{rmk}
\label{rmk:bracketone}
The proof of Lemma \ref{lem:bracketone} 
applies just as well to show that $\bracket{2^4, \la_3, \la_3}$ 
is detected by $\rho^{10} v_4$ in the Milnor-Witt 15-stem.  
In higher stems, the analogous brackets do not produce generators.
For example, 
the Massey product $\bracket{\rho^5, \rho^{10}v_4,\rho^{10}v_4}$ is already defined in $\Ext$, which implies that the corresponding Toda bracket 
must be detected in filtration least $27$.  However, $\rho^{25}v_5$ 
detects the generator of the Milnor-Witt 31-stem, and it 
lies in filtration $26$.
\end{rmk}

\begin{lemma}
\label{lem:lambdan-bracket}
For $n \geq 2$,
the Toda bracket
$\langle 2^3, \lambda_2, P^{2^{n-1} - 2} \lambda_2 \rangle$ 
is detected by $\rho^{2^{n+1} - n - 3} v_{n+1}$.
The indeterminacy in this Toda bracket
is generated by $2^3 \lambda_{n+1}$.
\end{lemma}

\begin{proof}
This follows from Moss's convergence theorem \cite[Theorem~1.2]{Moss},
together with the Adams differential
$d_n ( \rho^{2^{n+1} - n - 6} v_{n+1} ) = P^{2^{n-1} - 2} v_2^2$.
\end{proof}

\begin{lemma}
\label{lem:lambda2-bracket}
For $n \geq 2$,
the Toda bracket
$\langle 2^{n+2}, \lambda_{n+1}, P^{2^{n-1} - 2} \lambda_2 \rangle$
is detected by $P^{2^n-2} v_2$.
The Toda bracket has no indeterminacy.
\end{lemma}

\begin{proof}
Lemma \ref{lem:ExtMassey2} implies that there is a Massey product
\[
P^{2^n-2} v_2 = \langle \rho^{n+2}, \rho^{2^{n+1} - n - 3} v_{n+1},
P^{2^{n-1} - 2} v_2 \rangle,
\]
with no indeterminacy.
Moss's convergence theorem \cite[Theorem~1.2]{Moss}
establishes the desired result.
\end{proof}

\begin{lemma}\label{bracketthree} 
If $m>n \geq 2$, then
the Toda bracket $\bracket{2^{m+1},\la_m,P^{2^{n-1} k}\la_n}$ 
is detected by $\rho^{2^n - n - 2} P^{2^{m-2} + 2^{n-1} k} v_n$.
The Toda bracket has no indeterminacy.
\end{lemma}

\begin{pf}
Lemma \ref{lem:ExtMassey2} implies that there is a Massey product
\[  \rho^{2^n-n-2}P^{2^{m-2}+2^{n-1} k }v_n =  
\bracket{ \rho^{m+1}, \rho^{2^m-m-2}v_m, \rho^{2^n-n-2}P^{2^{n-1} k }v_n}.\]
Moss's convergence theorem \cite[Theorem~1.2]{Moss}
establishes the desired result.
\end{pf}

\begin{prop}
Every generator $P^{2^{n-1} k} \lambda_n$ of
the $\eta$-inverted $\R$-motivic stable homotopy ring can be constructed via
iterated 3-fold Toda brackets starting from $2$ and $\lambda_2$.
\end{prop}

\begin{proof}
 Lemmas \ref{lem:lambdan-bracket} and \ref{lem:lambda2-bracket}
alternately show that the generators $\lambda_n$
and the generators $P^{2^n-2} \lambda_2$ can be constructed
via iterated 3-fold Toda bracket starting from $2$ and $\lambda_2$.
Then Lemma \ref{bracketthree}
shows that any $P^{2^{n-1} k} \lambda_n$ can be constructed.
\end{proof}

\begin{eg} 
Suppose we wish to find a Toda bracket
decomposition for $P^{40} \lambda_3$.
Since $40 = 2^{7-2} + 2^{3-1} \cdot 4$, 
we can apply
Lemma \ref{bracketthree} 
with $m = 7$, $n = 3$, and $k = 4$ to conclude that
$P^{40} \lambda_3$ is detected by the Toda bracket
$\langle 2^{8}, \lambda_7, P^8 \lambda_3 \rangle$.
\end{eg}

\clearpage

\psset{linewidth=0.3mm}    
\psset{unit=2.6mm}

\begin{figure}
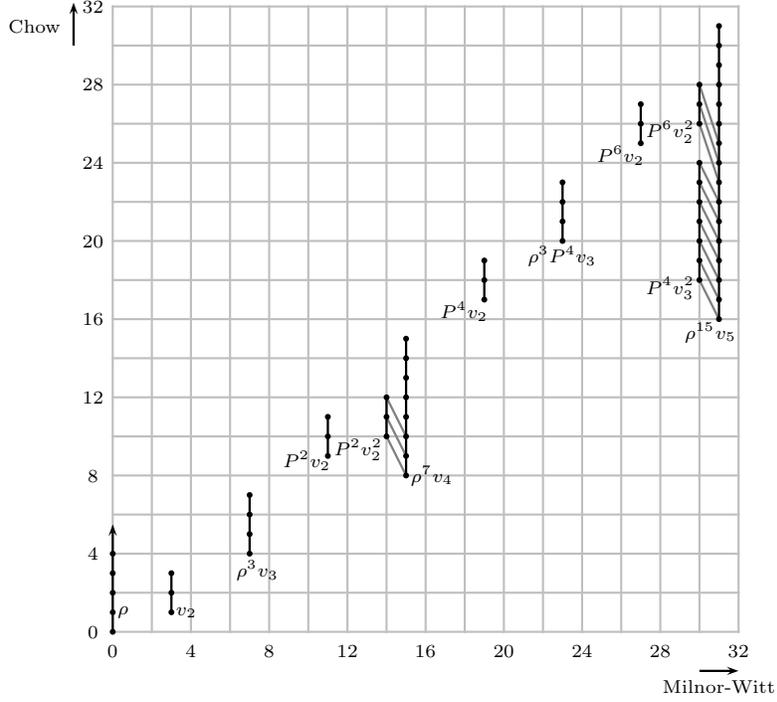


\caption{The $\eta$-inverted $\R$-motivic Adams $E_3$-page\label{fig:E3}}
\begin{pspicture}(-2,-4)(32,32.3)

\psgrid[unit=2,gridcolor=gridline,subgriddiv=0,gridlabels=0](0,0)(16,16)

\scriptsize 

\input{E3ChartUpto32data.tex}

\scriptsize 

\psline{->}(30,-2)(32,-2)
\rput(31,-2.8){Milnor-Witt}
\psline{->}(-2,30)(-2,32.2)
\rput(-4,31){Chow}

\end{pspicture}

\begin{pspicture}(30,30)(64,64.3)

\psgrid[unit=2,gridcolor=gridline,subgriddiv=0,gridlabels=0](16,16)(32,32)


\input{E3ChartUpto64data.tex}

\scriptsize 

\psline{->}(62,30)(64,30)
\rput(63,29.2){Milnor-Witt}
\psline{->}(30,62)(30,64)
\rput(28,63){Chow}

\end{pspicture}

\end{figure}

\begin{figure}
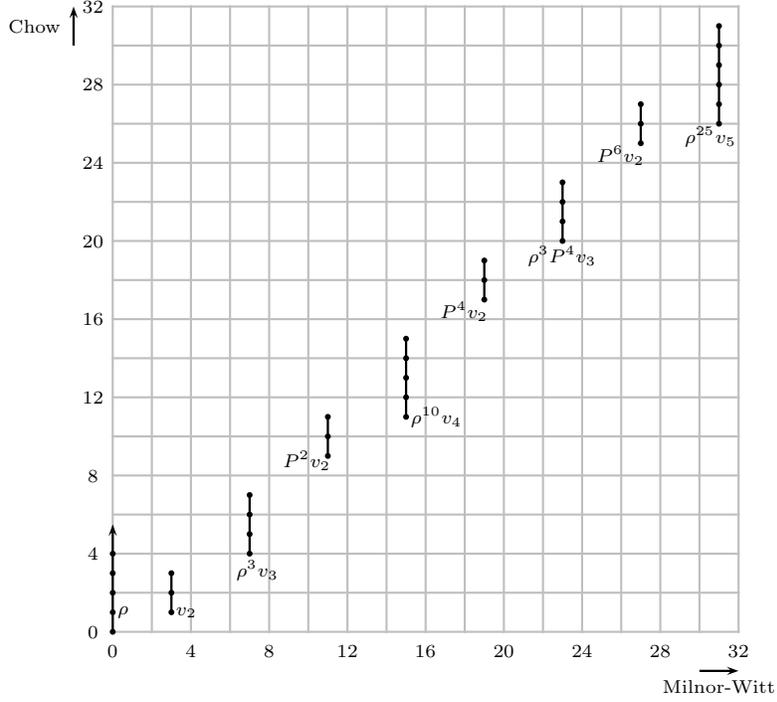


\caption{The $\eta$-inverted $\R$-motivic Adams $E_\infty$-page
\label{fig:Einfty}}
\begin{pspicture}(-2,-4)(32,32.3)

\psgrid[unit=2,gridcolor=gridline,subgriddiv=0,gridlabels=0](0,0)(16,16)

\scriptsize 

\input{EinfChartUpto32data.tex}

\scriptsize 

\psline{->}(30,-2)(32,-2)
\rput(31,-2.8){Milnor-Witt}
\psline{->}(-2,30)(-2,32)
\rput(-4,31){Chow}

\end{pspicture}

\begin{pspicture}(30,30)(64,64.3)

\psgrid[unit=2,gridcolor=gridline,subgriddiv=0,gridlabels=0](16,16)(32,32)


\input{EinfChartUpto64data.tex}

\scriptsize 

\psline{->}(62,30)(64,30)
\rput(63,29.2){Milnor-Witt}
\psline{->}(30,62)(30,64)
\rput(28,63){Chow}

\end{pspicture}

\end{figure}

\end{document}